\documentclass[reqno,12pt]{amsart}
\usepackage{amssymb}
\usepackage{amsfonts}
\usepackage{amsmath}
\usepackage{graphicx}
\usepackage{amscd,amsthm}

\newtheorem{theorem}{Theorem}
\theoremstyle{plain}
\newtheorem{acknowledgement}{Acknowledgement}

\newtheorem{definition}{Definition}

\numberwithin{equation}{section}

\input{tcilatex}

\begin{document}
\author{}
\title{}
\maketitle

\begin{center}
\thispagestyle{empty} \pagestyle{myheadings} 
\markboth{\bf Burak Kurt}{\bf
Identities and Relations On The Hermite-based Tangent Polynomials}

\textbf{\Large Identities And Relations On The Hermite-based Tangent
Polynomials}

\bigskip

\textbf{Burak Kurt} \bigskip

\medskip

Department of Mathematics, Faculty of Educations\\[0pt]

University of Akdeniz\\[0pt]

TR-07058 Antalya, Turkey \\[0pt]

E-mail\textbf{: burakkurt@akdeniz.edu.tr}

\medskip

\bigskip

\bigskip

\textbf{{\large {Abstract}}}\medskip
\end{center}

\begin{quotation}
In this note, we introduce and investigate the Hermite-based Tangent numbers
and polynomials, Hermite-based modified degenerate-Tangent polynomials,
poly-Tangent polynomials. We give some identities and relations for these
polynomials.
\end{quotation}



\section{Introduction and Notation}

Many mathematicians have studied in the area of the Bernoulli numbers and
polynomials, Euler number and polynomials, Genocchi numbers and polynomials,
poly-Bernoulli numbers and polynomials, poly-Euler numbers and polynomials,
poly-Genocchi numbers and polynomials, poly-Tangent numbers and polynomials,
Hermite polynomials, Hermite-based Bernoulli polynomials, Hermite-based
Tangent polynomials, modified degenerate Bernoulli polynomials, modified
degenerate Euler polynomials and modified degenerate Genocchi polynomials
(see \cite{1}-\cite{20}). In this note we define the Hermite-based tangent
polynomials, modified Hermite-based tangent polynomials and poly-tangent
polynomials. We obtain some relations and identities for these polynomials.
Throughout this paper, we always make use of the following notations: $%
\mathbb{N}
$ denotes the set of natural numbers and $%
\mathbb{Z}
_{+}=%
\mathbb{N}
\cup \left\{ 0\right\} $. We recall that the classical Stirling numbers of
the first kind $S_{1}(n,k)$ and second kind $S_{2}(n,k)$ are defined by the
relations \cite{15}%
\begin{equation}
\left( x\right) _{n}=\sum_{k=0}^{n}S_{1}(n,k)x^{k}\text{ \ \ and \ \ }%
x^{n}=\sum_{k=0}^{n}S_{2}(n,k)\left( x\right) _{k}\text{\ }  \label{1}
\end{equation}%
respectively. Here $\left( x\right) _{n}=x(x-1)\cdots (x-n+1)$ denotes the
falling factorial polynomials of order $n$. The numbers $S_{2}(n,m)$ also
admit a representation in terms of a generating function%
\begin{equation}
\frac{\left( e^{t}-1\right) ^{m}}{m!}=\sum_{n=m}^{\infty }S_{2}(n,m)\frac{%
t^{n}}{n!}\text{.}  \label{2}
\end{equation}

The Bernoulli polynomials $B_{n}^{\left( r\right) }(x)$ of order $\alpha $,
the Euler polynomials $E_{n}^{\left( r\right) }(x;\lambda )$ of order $%
\alpha $ and the Genocchi polynomials $G_{n}^{\left( r\right) }(x;\lambda )$
of order $\alpha $ are defined as respectively:%
\begin{equation}
\left( \frac{t}{e^{t}-1}\right) ^{r}e^{xt}=\sum\limits_{n=0}^{\infty
}B_{n}^{\left( r\right) }(x)\frac{t^{n}}{n!}\text{,}\text{ }\left\vert
t\right\vert <2\pi \text{,}  \label{3}
\end{equation}%
\begin{equation}
\left( \frac{2}{e^{t}+1}\right) ^{r}e^{xt}=\sum\limits_{n=0}^{\infty
}E_{n}^{\left( r\right) }(x)\frac{t^{n}}{n!}\text{, }\left\vert t\right\vert
<\pi  \label{4}
\end{equation}%
and%
\begin{equation}
\left( \frac{2t}{e^{t}+1}\right) ^{r}e^{xt}=\sum\limits_{n=0}^{\infty
}G_{n}^{\left( r\right) }(x)\frac{t^{n}}{n!}\text{, }\left\vert t\right\vert
<\pi \text{.}  \label{5}
\end{equation}%
When $x=0$, $B_{n}^{\left( r\right) }(0)=B_{n}^{\left( r\right) }$, $%
E_{n}^{\left( r\right) }(0)=E_{n}^{\left( r\right) }$ and $G_{n}^{\left(
r\right) }(0)=G_{n}^{\left( r\right) }$ are called Bernoulli numbers of
order $r$, Euler numbers of order $r$ and Genocchi numbers of order $r$,
respectively.

The familiar tangent polynomials $T_{n}^{\left( r\right) }\left( x\right) $
of order $r$ are defined by the generating functions (\cite{12}-\cite{15}, 
\cite{17})%
\begin{equation}
\left( \frac{2}{e^{2t}+1}\right) ^{r}e^{xt}=\sum\limits_{n=0}^{\infty
}T_{n}^{\left( r\right) }\left( x\right) \frac{t^{n}}{n!}\text{, }\left\vert
2t\right\vert <\pi \text{.}  \label{6}
\end{equation}%
When $x=0$, $T_{n}^{\left( r\right) }\left( 0\right) =T_{n}^{\left( r\right)
}$ are called the tangent numbers.

$2$-variable Hermite-Kamp\'{e}de F\'{e}riet polynomials are defined in (\cite%
{5}, \cite{11}) as%
\begin{equation}
\sum\limits_{n=0}^{\infty }H_{n}\left( x,y\right) \frac{t^{n}}{n!}%
=e^{xt+yt^{2}}\text{.}  \label{7}
\end{equation}

Khan \textit{et al.} in \cite{5} defined and studied on Hermite-based
Bernoulli polynomials and Hermite-based Euler polynomials as%
\begin{equation}
\sum\limits_{n=0}^{\infty }\left( _{H}\mathcal{B}_{n}(x,y)\right) \frac{t^{n}%
}{n!}=\frac{t}{e^{t}-1}e^{xt+yt^{2}}\text{,}\text{ }\left\vert t\right\vert
<2\pi  \label{8}
\end{equation}%
and%
\begin{equation}
\sum\limits_{n=0}^{\infty }\left( _{H}\mathcal{E}_{n}(x,y)\right) \frac{t^{n}%
}{n!}=\frac{2}{e^{t}+1}e^{xt+yt^{2}}\text{,}\text{ }\left\vert t\right\vert
<\pi \text{,}  \label{9}
\end{equation}%
respectively.

Carlitz in \cite{1} defined degenerate Bernoulli polynomials which are given
by the generating functions to be%
\begin{equation}
\frac{t}{\left( 1+\lambda t\right) ^{1/\lambda }-1}\left( 1+\lambda t\right)
^{x/\lambda }=\sum\limits_{n=0}^{\infty }\mathfrak{B}_{n}(x\mid \lambda )%
\frac{t^{n}}{n!}\text{.}  \label{10}
\end{equation}

When $x=0$, $\mathfrak{B}_{n}(\lambda )=\mathfrak{B}_{n}(0\mid \lambda )$
are called the degenerate Bernoulli numbers.

From (\ref{10}), we can easily derive the following equation%
\begin{equation*}
\mathfrak{B}_{n}(x\mid \lambda )=\sum\limits_{l=0}^{n}\binom{n}{l}\mathfrak{B%
}_{n-l}(\lambda )\left( x\mid \lambda \right) _{l}\text{, }n\geq 0
\end{equation*}%
where $\left( x\mid \lambda \right) _{n}=x\left( x-\lambda \right) \cdots
\left( x-\lambda \left( n-1\right) \right) $, $\left( x\mid \lambda \right)
_{n}=1$.

Dolgy \textit{et. al.} \cite{2} defined the modified degenerate Bernoulli
polynomials, which are different from Carlitz's degenerate Bernoulli
polynomials as%
\begin{equation}
\frac{t}{\left( 1+\lambda \right) ^{t/\lambda }-1}\left( 1+\lambda \right)
^{xt/\lambda }=\sum\limits_{n=0}^{\infty }\mathfrak{B}_{n,\lambda }(x)\frac{%
t^{n}}{n!}\text{.}  \label{11}
\end{equation}%
When $x=0$, $\mathfrak{B}_{n,\lambda }=\mathfrak{B}_{n,\lambda }(0)$ are
called the modified degenerate Bernoulli numbers. From (\ref{11}), we note
that%
\begin{eqnarray}
\underset{\lambda \rightarrow 0}{\lim }\sum\limits_{n=0}^{\infty }\mathfrak{B%
}_{n,\lambda }(x)\frac{t^{n}}{n!} &=&\underset{\lambda \rightarrow 0}{\lim }%
\frac{t}{\left( 1+\lambda \right) ^{t/\lambda }-1}\left( 1+\lambda \right)
^{xt/\lambda }  \notag \\
&=&\frac{t}{e^{t}-1}e^{xt}=\sum\limits_{n=0}^{\infty }B_{n}(x)\frac{t^{n}}{n!%
}\text{.}  \label{12}
\end{eqnarray}

Thus, by (\ref{12})%
\begin{equation*}
\underset{\lambda \rightarrow 0}{\lim }\mathfrak{B}_{n,\lambda }(x)=B_{n}(x)%
\text{.}
\end{equation*}

H.-In Known \textit{et. al.} \cite{8} defined the modified degenerate Euler
polynomials as%
\begin{equation}
\frac{2}{\left( 1+\lambda \right) ^{t/\lambda }+1}\left( 1+\lambda \right)
^{xt/\lambda }=\sum\limits_{n=0}^{\infty }\mathfrak{E}_{n,\lambda }(x)\frac{%
t^{n}}{n!}  \label{13}
\end{equation}%
and T. Kim \textit{et. al.} in \cite{6} defined the modified degenerate
Genocchi polynomials as%
\begin{equation}
\frac{2t}{\left( 1+\lambda \right) ^{t/\lambda }+1}\left( 1+\lambda \right)
^{tx/\lambda }=\sum\limits_{n=0}^{\infty }\mathfrak{G}_{n,\lambda }(x)\frac{%
t^{n}}{n!}\text{.}  \label{14}
\end{equation}%
From (\ref{13}) and (\ref{14}), we get%
\begin{equation*}
\underset{\lambda \rightarrow 0}{\lim }\mathfrak{E}_{n,\lambda }(x)=E_{n}(x)%
\text{, }\underset{\lambda \rightarrow 0}{\lim }\mathfrak{G}_{n,\lambda
}(x)=G_{n}(x)\text{.}
\end{equation*}%
For $k\in 
\mathbb{Z}
$, $k>1$, then $k$-th polylogarithm is defined by Kaneko \cite{4} as%
\begin{equation}
L_{i_{k}}(z)=\sum\limits_{n=1}^{\infty }\frac{z^{n}}{n^{k}}\text{.}
\label{15}
\end{equation}%
Thus function is convergent for $\left\vert z\right\vert <1$, when $k=1$%
\begin{equation}
L_{i_{1}}(z)=-\log (1-z)\text{.}  \label{16}
\end{equation}

Kim \textit{et. al.} in \cite{7} defined the poly-Bernoulli polynomials and
the poly-Genocchi polynomials as%
\begin{equation}
\sum\limits_{n=0}^{\infty }\mathfrak{B}_{n}^{\left( k\right) }(x)\frac{t^{n}%
}{n!}=\frac{L_{i_{k}}(1-e^{-t})}{1-e^{-t}}e^{xt}  \label{17}
\end{equation}%
and%
\begin{equation}
\sum\limits_{n=0}^{\infty }\mathfrak{G}_{n}^{\left( k\right) }(x)\frac{t^{n}%
}{n!}=\frac{2L_{i_{k}}(1-e^{-t})}{e^{t}+1}e^{xt}  \label{18}
\end{equation}%
respectively.

For $k=1$, by use (\ref{16}) in (\ref{17}) and (\ref{18}). We get%
\begin{equation*}
\mathfrak{B}_{n}^{\left( 1\right) }(x)=(-1)^{n+1}B_{n}(x)\text{, }\mathfrak{G%
}_{n}^{\left( 1\right) }(x)=G_{n}(x)\text{.}
\end{equation*}

Hamahata \cite{3} defined poly-Euler polynomials by%
\begin{equation*}
\sum\limits_{n=0}^{\infty }\mathfrak{E}_{n}^{\left( k\right) }(x)\frac{t^{n}%
}{n!}=\frac{2L_{i_{k}}(1-e^{-t})}{t\left( e^{t}+1\right) }e^{xt}\text{.}
\end{equation*}%
For $k=1$, we get $\mathfrak{E}_{n}^{\left( 1\right) }(x)=E_{n}(x)$.

From (\ref{6}), we obtain the following equalities easily%
\begin{equation*}
T_{n}^{\left( r\right) }\left( x\right) =\sum\limits_{k=0}^{n}\binom{n}{k}%
T_{k}^{\left( r\right) }x^{n-k}\text{,}
\end{equation*}%
\begin{eqnarray*}
T_{n}^{\left( r\right) }\left( x+y\right) &=&\sum\limits_{l=0}^{k}\binom{k}{l%
}T_{k}^{\left( r\right) }(x)y^{k-l}\text{,} \\
T_{n}^{\left( r_{1}+r_{2}\right) }\left( x+y\right) &=&\sum\limits_{k=0}^{n}%
\binom{n}{k}T_{k}^{\left( r_{1}\right) }(x)T_{n-k}^{\left( r_{2}\right) }(y)
\end{eqnarray*}%
and%
\begin{equation*}
T_{n}^{\left( r\right) }\left( 2\left( x+1\right) \right) =2T_{n}^{\left(
r-1\right) }\left( 2x\right) \text{.}
\end{equation*}

\section{Hermite Based Tangent Polynomials}

Khan \textit{et. al.} in \cite{5} and Ozarslan \cite{11} introduced and
investigated the Hermite-based Bernoulli polynomials and Hermite-based Euler
polynomials. They proved some identities and relations for these polynomials.

By this motivation, we define Hermite-based Tangent polynomials of order $r$
as%
\begin{equation}
\sum\limits_{n=0}^{\infty }\left( _{H}T_{n}^{\left( r\right) }\left(
x,y\right) \right) \frac{t^{n}}{n!}=\left( \frac{2}{e^{2t}+1}\right)
^{r}e^{xt+yt^{2}}\text{.}  \label{19}
\end{equation}

\begin{theorem}
Let $r_{1}$, $r_{2}\in 
\mathbb{Z}
_{+}$. We have%
\begin{eqnarray*}
_{H}T_{n}^{\left( r\right) }\left( x,y\right) &=&\sum\limits_{k=0}^{n}\binom{%
n}{k}T_{k}^{\left( r\right) }\left( 0,0\right) H_{n-k}(x,y)\text{,} \\
_{H}T_{n}^{\left( r\right) }\left( x+u,y+v\right) &=&\sum\limits_{k=0}^{n}%
\binom{n}{k}\left( _{H}T_{k}^{\left( r\right) }\left( x,y\right) \right)
H_{n-k}(u,v)
\end{eqnarray*}%
and%
\begin{equation*}
_{H}T_{n}^{\left( r_{1}+r_{2}\right) }\left( x+u,y+v\right)
=\sum\limits_{k=0}^{n}\binom{n}{k}\left( _{H}T_{k}^{\left( r_{1}\right)
}\left( x,y\right) \right) \left( _{H}T_{n-k}^{\left( r_{2}\right) }\left(
u,v\right) \right) \text{.}
\end{equation*}
\end{theorem}

\begin{theorem}
Let $r\in 
\mathbb{Z}
_{+}$.%
\begin{equation*}
_{H}T_{n}^{\left( r\right) }\left( 2\left( x+u\right) ,2\left( y+v\right)
\right) =\sum\limits_{m=0}^{n}\binom{n}{m}\left( _{H}T_{n-m}^{\left(
r\right) }\left( x,y\right) \right) \sum\limits_{p=0}^{m}\binom{m}{p}%
H_{p}(x,y)H_{m-p}(x,y)\text{.}
\end{equation*}
\end{theorem}

\begin{theorem}
There is the following implicit relation for the Hermite-based Tangent
polynomials as%
\begin{equation}
_{H}T_{n+m}^{\left( r\right) }\left( u,v\right) =\sum\limits_{p=0}^{n}\binom{%
n}{p}\sum\limits_{q=0}^{m}\binom{m}{q}\left( v-y\right) ^{p+q}\left(
_{H}T_{n+m-p-q}^{\left( r\right) }\left( x,y\right) \right) \text{.}
\label{20}
\end{equation}
\end{theorem}

\begin{proof}
From (\ref{19}), we replace $t$ by $t+u$ and rewrite the generating function
as%
\begin{equation*}
\frac{2e^{y\left( t+u\right) ^{2}}}{e^{2t}+1}=e^{-x\left( t+u\right)
}\sum\limits_{n=0}^{\infty }T_{n+m}^{\left( r\right) }\left( x,y\right) 
\frac{t^{n}}{n!}\frac{u^{m}}{m!}\text{.}
\end{equation*}%
Replacing $x$ by $v$ in the above equation to the above equation.

We get%
\begin{equation*}
\sum\limits_{n,m=0}^{\infty }\left( _{H}T_{n+m}^{\left( r\right) }\left(
v,y\right) \right) \frac{t^{n}}{n!}\frac{u^{m}}{m!}=e^{\left( t+u\right)
\left( v-x\right) }\sum\limits_{n,m=0}^{\infty }\left( _{H}T_{n+m}^{\left(
r\right) }\left( x,y\right) \right) \frac{t^{n}}{n!}\frac{u^{m}}{m!}
\end{equation*}%
which on using formula \cite[Srivastava p. 52]{19}%
\begin{equation*}
\sum\limits_{N=0}^{\infty }f\left( N\right) \frac{\left( x+y\right) ^{N}}{N!}%
=\sum\limits_{n,m=0}^{\infty }f\left( n+m\right) \frac{x^{n}}{n!}\frac{y^{m}%
}{m!}
\end{equation*}%
in the right hand side becomes%
\begin{equation*}
\sum\limits_{p=0}^{\infty }\sum\limits_{q=0}^{\infty }\left( v-x\right)
^{p+q}\frac{t^{p}}{p!}\frac{u^{q}}{q!}\sum\limits_{n=0}^{\infty
}\sum\limits_{m=0}^{\infty }\left( _{H}T_{n+m}^{\left( r\right) }\left(
x,y\right) \right) \frac{t^{n}}{n!}\frac{u^{m}}{m!}=\sum\limits_{n,m=0}^{%
\infty }\left( _{H}T_{n+m}^{\left( r\right) }\left( v,y\right) \right) \frac{%
t^{n}}{n!}\frac{u^{m}}{m!}\text{.}
\end{equation*}%
By using Cauchy product and comparing the coefficients of both sides, we
have (\ref{20}).
\end{proof}

\begin{theorem}
There is the following relation between the Hermite-based Tangent
polynomials and the Hermite-based Bernoulli polynomials as%
\begin{equation}
\left( _{H}\mathfrak{B}_{n}^{\left( r\right) }\left( \frac{x+u}{4},\frac{y+v%
}{16}\right) \right) =2^{r-n-k}\sum\limits_{k=0}^{n}\binom{n}{k}\left(
_{H}T_{k}^{\left( r\right) }\left( x,y\right) \right) \left( _{H}\mathfrak{B}%
_{n-k}^{\left( r\right) }\left( \frac{u}{2},\frac{v}{4}\right) \right) \text{%
.}  \label{21}
\end{equation}
\end{theorem}

\begin{proof}
From (\ref{19}),%
\begin{equation*}
\sum\limits_{n=0}^{\infty }\left( _{H}\mathfrak{B}_{n}^{\left( r\right)
}\left( \frac{x+u}{4},\frac{y+v}{16}\right) \right) \frac{\left( 4t\right)
^{n}}{n!}=\left( \frac{2\times 4t}{e^{4t}-1}\right) ^{\left( r\right)
}e^{\left( x+u\right) t+\left( y+v\right) t^{2}}
\end{equation*}%
\begin{eqnarray*}
&=&\left( \frac{2}{e^{2t}+1}\right) ^{\left( r\right)
}e^{xt+yt^{2}}2^{r}\left( \frac{2t}{e^{2t}-1}\right) ^{\left( r\right)
}e^{ut+vt^{2}} \\
&=&\sum\limits_{n=0}^{\infty }\left( _{H}T_{n}^{\left( r\right) }\left(
x,y\right) \right) \frac{t^{n}}{n!}2^{r}\sum\limits_{q=0}^{\infty }\left(
_{H}\mathfrak{B}_{q}^{\left( r\right) }\left( \frac{u}{2},\frac{v}{4}\right)
\right) \frac{\left( 2t\right) ^{n}}{n!}\text{.}
\end{eqnarray*}%
By using Cauchy product and comparing the coefficients of both sides. We get
(\ref{21}).
\end{proof}

\section{Modified Degenerate Hermite-Based Tangent Polynomials}

Dolgy \textit{et. al.} \cite{2} introduced and investigated the modified
degenerate Bernoulli polynomials. Known \textit{et. al.} \cite{8} defined
and investigated the modified degenerate Euler polynomials. They proved some
properties for these polynomials.

By these motivations, we define $2$-variable modified degenerate Hermite
polynomials and the modified degenerate Hermite-based Tangent polynomials of
order $r$%
\begin{equation}
\sum\limits_{n=0}^{\infty }H_{n}\left( x,y:\lambda \right) \frac{t^{n}}{n!}%
=\left( 1+\lambda \right) ^{\frac{xt+yt^{2}}{\lambda }}  \label{22}
\end{equation}%
and%
\begin{equation}
\sum\limits_{n=0}^{\infty }\left( _{H}T_{n}^{\left( r\right) }\left(
x,y:\lambda \right) \right) \frac{t^{n}}{n!}=\frac{2}{\left( 1+\lambda
\right) ^{\frac{2t}{\lambda }}+1}\left( 1+\lambda \right) ^{\frac{xt+yt^{2}}{%
\lambda }}  \label{23}
\end{equation}%
respectively.

From (\ref{22}) and (\ref{23}), we get%
\begin{equation*}
\underset{\lambda \longrightarrow 0}{\lim }H_{n}\left( x,y:\lambda \right)
=H_{n}\left( x,y\right) \text{, }\underset{\lambda \longrightarrow 0}{\lim }%
\left( _{H}T_{n}^{\left( r\right) }\left( x,y:\lambda \right) \right)
=\left( _{H}T_{n}^{\left( r\right) }\left( x,y\right) \right) \text{.}
\end{equation*}

Similiary, we define the modified Hermite-based Bernoulli poynomials and the
modified Hermite-based Euler polynomials as%
\begin{equation}
\sum\limits_{n=0}^{\infty }\left( _{H}\mathfrak{B}_{n}\left( x,y:\lambda
\right) \right) \frac{t^{n}}{n!}=\frac{t}{\left( 1+\lambda \right) ^{\frac{t%
}{\lambda }}-1}\left( 1+\lambda \right) ^{\frac{xt+yt^{2}}{\lambda }}
\label{24}
\end{equation}%
and%
\begin{equation}
\sum\limits_{n=0}^{\infty }\left( _{H}\mathfrak{E}_{n}\left( x,y:\lambda
\right) \right) \frac{t^{n}}{n!}=\frac{2}{\left( 1+\lambda \right) ^{\frac{t%
}{\lambda }}+1}\left( 1+\lambda \right) ^{\frac{xt+yt^{2}}{\lambda }}
\label{25}
\end{equation}%
respectively.

From (\ref{23}), we obtain the following relations easily%
\begin{equation*}
\left( _{H}T_{n}^{\left( r_{1}+r_{2}\right) }\left( x+u,y+v:\lambda \right)
\right) =\sum\limits_{k=0}^{n}\binom{n}{k}\left( _{H}T_{k}^{\left(
r_{1}\right) }\left( x,y:\lambda \right) \right) \left( _{H}T_{n-k}^{\left(
r_{2}\right) }\left( u,v:\lambda \right) \right) \text{,}
\end{equation*}%
\begin{equation*}
\left( _{H}T_{n}^{\left( r\right) }\left( x,y:\lambda \right) \right)
=\sum\limits_{k=0}^{n}\binom{n}{k}\left( _{H}T_{k}^{\left( r\right) }\left(
0,0:\lambda \right) \right) \left( H_{n-k}\left( x,y:\lambda \right) \right) 
\text{,}
\end{equation*}%
\begin{equation*}
\left( _{H}T_{n}^{\left( r\right) }\left( x+2,y:\lambda \right) \right)
+\left( _{H}T_{n}^{\left( r\right) }\left( x,y:\lambda \right) \right)
=2\left( _{H}T_{n}^{\left( r-1\right) }\left( x,y:\lambda \right) \right)
\end{equation*}%
for $r=1$,%
\begin{equation*}
\left( _{H}T_{n}\left( x+2,y:\lambda \right) \right) +\left( _{H}T_{n}\left(
x,y:\lambda \right) \right) =2\left( H_{n}\left( x,y:\lambda \right) \right)
\end{equation*}

and 
\begin{equation*}
\left( _{H}T_{n}^{\left( r\right) }\left( x,y:\lambda \right) \right)
=\sum\limits_{k=0}^{n}\binom{n}{k}\left( _{H}T_{n}^{\left( r\right) }\left( 
\frac{1}{2},0:\lambda \right) \right) \left( H_{n-k}\left( x-\frac{1}{2}%
,y:\lambda \right) \right) \text{.}
\end{equation*}

\begin{theorem}
There is the following relation between the modified degenerate Bernoulli
polynomials, the modified degenerate Euler polynomials and the modified
degenerate Tangent polynomials as%
\begin{equation}
\left( _{H}\mathfrak{B}_{n}\left( x,y:\lambda \right) \right)
2^{2n+1}=\sum\limits_{q=0}^{n}\binom{n}{q}\left( _{H}T_{n-q}\left(
x,y:\lambda \right) \right) \sum\limits_{k=0}^{q}\binom{q}{k}\left( _{H}%
\mathfrak{B}_{q-k}\left( x,y:\lambda \right) \right) \left( _{H}\mathfrak{E}%
_{n}\left( 2x,14y:\lambda \right) \right) \text{.}  \label{26}
\end{equation}
\end{theorem}

\begin{proof}
From (\ref{24}), (\ref{25}) and (\ref{23}), we write as%
\begin{equation*}
\sum\limits_{n=0}^{\infty }\left( _{H}\mathfrak{B}_{n}\left( x,y:\lambda
\right) \right) \frac{\left( 4t\right) ^{n}}{n!}=\left( \frac{4t}{\left(
1+\lambda \right) ^{\frac{4t}{\lambda }}-1}\right) \left( 1+\lambda \right)
^{\frac{4tx+y\left( 4t\right) ^{2}}{\lambda }}
\end{equation*}%
\begin{eqnarray*}
&=&\frac{1}{2}\frac{2e^{\frac{xt+yt^{2}}{\lambda }}}{\left( 1+\lambda
\right) ^{\frac{2t}{\lambda }}+1}\frac{2te^{\frac{xt+yt^{2}}{\lambda }}}{%
\left( 1+\lambda \right) ^{\frac{t}{\lambda }}-1}\frac{2e^{\frac{2xt+14yt^{2}%
}{\lambda }}}{\left( 1+\lambda \right) ^{\frac{t}{\lambda }}+1} \\
&=&\frac{1}{2}\sum\limits_{n=0}^{\infty }\left( _{H}T_{n}\left( x,y:\lambda
\right) \right) \frac{t^{n}}{n!}\sum\limits_{p=0}^{\infty }\left( _{H}%
\mathfrak{B}_{p}\left( x,y:\lambda \right) \right) \frac{t^{p}}{p!}%
\sum\limits_{q=0}^{\infty }\left( _{H}\mathfrak{E}_{q}\left( 2x,14y:\lambda
\right) \right) \frac{t^{q}}{q!}\text{.}
\end{eqnarray*}%
By using Cauchy product and comparing the coefficient of $\frac{t^{n}}{n!}$,
we have (\ref{26}).
\end{proof}

\begin{theorem}
$n\in 
\mathbb{Z}
_{+}$, we have%
\begin{equation}
\left( _{H}T_{n}\left( x+2,y:\lambda \right) \right) +\left( _{H}T_{n}\left(
x,y:\lambda \right) \right) =\frac{1}{n+1}\left\{ \left( _{H}\mathfrak{B}%
_{n+1}\left( x+1,y:\lambda \right) \right) -\left( _{H}\mathfrak{B}%
_{n+1}\left( x,y:\lambda \right) \right) \right\} \text{.}  \label{27}
\end{equation}
\end{theorem}

\begin{proof}
By using definition (\ref{23})%
\begin{equation*}
\frac{2t\left( 1+\lambda \right) ^{\frac{xt+yt^{2}}{\lambda }}}{\left(
1+\lambda \right) ^{\frac{2t}{\lambda }}+1}\left[ \left( 1+\lambda \right) ^{%
\frac{2t}{\lambda }}+1\right] =\frac{2t\left( 1+\lambda \right) ^{\frac{%
xt+yt^{2}}{\lambda }}}{\left( 1+\lambda \right) ^{\frac{t}{\lambda }}-1}%
\left[ \left( 1+\lambda \right) ^{\frac{t}{\lambda }}-1\right]
\end{equation*}%
\begin{equation*}
\frac{2t\left( 1+\lambda \right) ^{\frac{\left( x+2\right) t+yt^{2}}{\lambda 
}}}{\left( 1+\lambda \right) ^{\frac{2t}{\lambda }}+1}+\frac{2t\left(
1+\lambda \right) ^{\frac{xt+yt^{2}}{\lambda }}}{\left( 1+\lambda \right) ^{%
\frac{2t}{\lambda }}+1}=\frac{2t\left( 1+\lambda \right) ^{\frac{\left(
x+1\right) t+yt^{2}}{\lambda }}}{\left( 1+\lambda \right) ^{\frac{t}{\lambda 
}}-1}-\frac{2t\left( 1+\lambda \right) ^{\frac{xt+yt^{2}}{\lambda }}}{\left(
1+\lambda \right) ^{\frac{t}{\lambda }}-1}
\end{equation*}%
\begin{equation*}
t\sum\limits_{n=0}^{\infty }\left\{ \left( _{H}T_{n}\left( x+2,y:\lambda
\right) \right) +\left( _{H}T_{n}\left( x,y:\lambda \right) \right) \right\} 
\frac{t^{n}}{n!}=\sum\limits_{n=0}^{\infty }\left\{ \left( _{H}\mathfrak{B}%
_{n}\left( x+1,y:\lambda \right) \right) -\left( _{H}\mathfrak{B}_{n}\left(
x,y:\lambda \right) \right) \right\} \frac{t^{n}}{n!}\text{.}
\end{equation*}%
From the above equality, we have (\ref{27}).
\end{proof}

\section{Poly-Tangent Polynomials}

In this section, we define\ the poly-tangent numbers and polynomials and
provide some of their relevant properties.

\begin{definition}
We define the Hermite-based poly-tangent polynomials by%
\begin{equation}
\frac{2L_{i_{k}}\left( 1-e^{-t}\right) }{t\left( e^{2t}+1\right) }%
e^{xt+yt^{2}}=\sum\limits_{n=0}^{\infty }\left( _{H}\mathcal{T}_{n}^{\left(
k\right) }\left( x,y\right) \right) \frac{t^{n}}{n!}\text{,}  \label{28}
\end{equation}%
when $x=0$, $\left( _{H}\mathcal{T}_{n}^{\left( k\right) }\right) :=\left(
_{H}\mathcal{T}_{n}^{\left( k\right) }\left( 0,0\right) \right) $ are called
the Hermite-based poly-tangent numbers.
\end{definition}

For $k=1$ and $L_{i_{k}}\left( z\right) =-\log \left( 1-z\right) $, from (%
\ref{28})%
\begin{equation}
\frac{2L_{i_{1}}\left( 1-e^{-t}\right) }{t\left( e^{2t}+1\right) }%
e^{xt+yt^{2}}=\frac{2e^{xt+yt^{2}}}{e^{2t}+1}=\sum\limits_{n=0}^{\infty
}\left( _{H}\mathcal{T}_{n}\left( x,y\right) \right) \frac{t^{n}}{n!}\text{.}
\label{29}
\end{equation}

By (\ref{29}), we get%
\begin{equation*}
\left( _{H}\mathcal{T}_{n}^{\left( 1\right) }\left( x,y\right) \right)
=\left( _{H}T_{n}\left( x,y\right) \right) \text{.}
\end{equation*}

\begin{theorem}
$n$, $k\in 
\mathbb{Z}
_{+}$, we have%
\begin{equation}
\left( _{H}\mathcal{T}_{n}^{\left( k\right) }\left( x,y\right) \right) =%
\frac{1}{n+1}\sum\limits_{m=0}^{\infty }\frac{1}{\left( m+1\right) ^{k}}%
\sum\limits_{j=0}^{m+1}\left( -1\right) ^{j}\binom{m+1}{j}\left( _{H}%
\mathcal{T}_{n+1}\left( x-j,y\right) \right) \text{.}  \label{30}
\end{equation}
\end{theorem}

\begin{proof}
\begin{equation*}
\sum\limits_{n=0}^{\infty }\left( _{H}\mathcal{T}_{n}^{\left( k\right)
}\left( x,y\right) \right) \frac{t^{n}}{n!}=2\sum\limits_{m=0}^{\infty }%
\frac{\left( 1-e^{-t}\right) ^{m+1}}{\left( m+1\right) ^{k}}\frac{%
e^{xt+yt^{2}}}{t\left( e^{2t}+1\right) }
\end{equation*}%
\begin{eqnarray*}
&=&2\sum\limits_{m=0}^{\infty }\frac{1}{\left( m+1\right) ^{k}}%
\sum\limits_{j=0}^{m+1}\left( -1\right) ^{j}\binom{m+1}{j}\frac{%
e^{-tj+xt+yt^{2}}}{t\left( e^{2t}+1\right) } \\
&=&\sum\limits_{m=0}^{\infty }\frac{1}{\left( m+1\right) ^{k}}%
\sum\limits_{j=0}^{m+1}\left( -1\right) ^{j}\binom{m+1}{j}\frac{1}{t}\frac{2%
}{e^{2t}+1}e^{t\left( x-j\right) +yt^{2}} \\
&=&\sum\limits_{m=0}^{\infty }\frac{1}{\left( m+1\right) ^{k}}%
\sum\limits_{j=0}^{m+1}\left( -1\right) ^{j}\binom{m+1}{j}%
\sum\limits_{n=0}^{\infty }\left( _{H}\mathcal{T}_{n}\left( x-j,y\right)
\right) \frac{t^{n-1}}{n!}
\end{eqnarray*}%
\begin{equation*}
=\sum\limits_{m=0}^{\infty }\frac{1}{\left( m+1\right) ^{k}}%
\sum\limits_{j=0}^{m+1}\left( -1\right) ^{j}\binom{m+1}{j}%
\sum\limits_{n=-1}^{\infty }\frac{\left( _{H}\mathcal{T}_{n+1}\left(
x-j,y\right) \right) }{n+1}\frac{t^{n}}{n!}\text{.}
\end{equation*}%
Comparing the coefficients both sides, we have (\ref{30}).
\end{proof}

\begin{theorem}
There is the following relation between poly-tangent polynomials and the
Stirling numbers of the second kind and the Hermite-based Bernoulli
polynomials as%
\begin{equation}
\left( _{H}\mathcal{T}_{n}^{\left( k\right) }\left( x,y\right) \right)
=\sum\limits_{l=0}^{n}\frac{\binom{n}{l}}{\binom{l+r}{r}}S_{2}\left(
l+r,r\right) \sum\limits_{i=0}^{n-l}\left( _{H}\mathfrak{B}_{i}^{\left(
r\right) }\left( x,y\right) \right) \left( _{H}\mathcal{T}_{n-l-i}^{\left(
r\right) }\right) \text{.}  \label{31}
\end{equation}
\end{theorem}

\begin{proof}
From (\ref{28}), we write as%
\begin{eqnarray*}
\sum\limits_{n=0}^{\infty }\left( _{H}\mathcal{T}_{n}^{\left( k\right)
}\left( x,y\right) \right) \frac{t^{n}}{n!} &=&\frac{2L_{i_{k}}\left(
1-e^{-t}\right) }{t\left( e^{2t}+1\right) }e^{xt+yt^{2}} \\
&=&\frac{\left( e^{t}-1\right) ^{r}}{r!}\frac{r!}{t^{r}}\left( \frac{t}{%
e^{t}-1}\right) ^{r}e^{xt+yt^{2}}\frac{2L_{i_{k}}\left( 1-e^{-t}\right) }{%
t\left( e^{2t}+1\right) }
\end{eqnarray*}%
\begin{eqnarray*}
&=&\frac{\left( e^{t}-1\right) ^{r}}{r!}\left( \sum\limits_{n=0}^{\infty
}\left( _{H}\mathfrak{B}_{n}^{\left( r\right) }\left( x,y\right) \right) 
\frac{t^{n}}{n!}\right) \left( \sum\limits_{q=0}^{\infty }\left( _{H}%
\mathcal{T}_{q}^{\left( r\right) }\right) \frac{t^{q}}{q!}\right) \frac{r!}{%
t^{r}} \\
&=&\sum\limits_{n=0}^{\infty }\left( \sum\limits_{l=0}^{n}\frac{\binom{n}{l}%
}{\binom{l+r}{r}}S_{2}\left( l+r,r\right) \sum\limits_{i=0}^{n-l}\left( _{H}%
\mathfrak{B}_{i}^{\left( r\right) }\left( x,y\right) \right) \left( _{H}%
\mathcal{T}_{n-l-i}^{\left( r\right) }\right) \right) \frac{t^{n}}{n!}\text{.%
}
\end{eqnarray*}%
Comparing the coefficients of $\frac{t^{n}}{n!}$, we obtain (\ref{31}).
\end{proof}

\begin{theorem}
There is the following relation between the poly-tangent polynomials, the
poly-Genocchi numbers and the Hermite-based tangent polynomials%
\begin{equation}
\left( _{H}\mathcal{T}_{n}^{\left( k\right) }\left( x,y\right) \right) =%
\frac{1}{2}\sum\limits_{p=0}^{n}\binom{n}{p}G_{n-p}^{\left( k\right)
}\left\{ \left( _{H}\mathcal{T}_{n}\left( x+1,y\right) \right) +\left( _{H}%
\mathcal{T}_{n}\left( x,y\right) \right) \right\} \text{.}  \label{32}
\end{equation}
\end{theorem}

\begin{proof}
From (\ref{28}) and (\ref{18})%
\begin{equation*}
\sum\limits_{n=0}^{\infty }\left( _{H}\mathcal{T}_{n}^{\left( k\right)
}\left( x,y\right) \right) \frac{t^{n}}{n!}=\frac{2L_{i_{k}}\left(
1-e^{-t}\right) }{t\left( e^{2t}+1\right) }e^{xt+yt^{2}}
\end{equation*}%
\begin{eqnarray*}
&=&\frac{1}{2}\left( \frac{2L_{i_{k}}\left( 1-e^{-t}\right) }{e^{t}+1}%
\right) \frac{2\left( e^{t}+1\right) e^{xt+yt^{2}}}{t\left( e^{2t}+1\right) }
\\
&=&\frac{1}{2}\frac{2L_{i_{k}}\left( 1-e^{-t}\right) }{e^{t}+1}\left( \frac{%
2e^{\left( x+1\right) t+yt^{2}}}{t\left( e^{2t}+1\right) }+\frac{%
2e^{xt+yt^{2}}}{t\left( e^{2t}+1\right) }\right) \\
&=&\frac{1}{2}\sum\limits_{n=0}^{\infty }G_{n}^{\left( k\right) }\frac{t^{n}%
}{n!}\left\{ \sum\limits_{p=0}^{\infty }\left( _{H}\mathcal{T}_{p}\left(
x+1,y\right) \right) +\left( _{H}\mathcal{T}_{p}\left( x,y\right) \right) 
\frac{t^{p}}{p!}\right\} \text{.}
\end{eqnarray*}%
Comparing the coefficients of both sides, we have (\ref{32}).
\end{proof}

\begin{acknowledgement}
The present investigation was supported, by the \textit{Scientific Research
Project Administration of Akdeniz University}.
\end{acknowledgement}

\end{document}